\documentclass[12pt]{article}

\usepackage{amsmath,amssymb,amsfonts,enumerate,amsthm, amscd}
\usepackage[french, ngerman, italian, english]{babel}
\usepackage{amsmath,amsthm, amssymb}
\usepackage{setspace}
\usepackage{anysize}
\usepackage{url}
\usepackage{xcolor}
\usepackage{graphicx}
\usepackage[colorlinks=true]{hyperref}
\usepackage{mathptmx}
\usepackage{paralist}
\usepackage{verbatim}

\theoremstyle{plain}
\newtheorem{thm}{\bf Theorem}[section]

\newtheorem{prop}[thm]{\bf Proposition}
\newtheorem{lem}[thm]{\bf Lemma}
\newtheorem{cor}[thm]{\bf Corollary}

\theoremstyle{definition}
\newtheorem{defn}[thm]{\bf Definition}
\theoremstyle{remark}
\newtheorem{rem}[thm]{\bf Remark}
\newtheorem{exam}[thm]{\bf Example}
\theoremstyle{example}

\def \pd{\mathrm{proj.dim}}
\def \dim{\mathrm{dim}}
\def \depth{\mathrm{depth}}

\def \1{\mathbf 1}

\def \x{\mathbf x}
\def \a{\mathbf a}

\def \u{\mathbf u}
\def \v{\mathbf v}
\def \z{\mathbf z}
\def \w{\mathbf w}
\def \s{\mathrm{sort}}

\def \rank{\mathrm{rank}}
\def \inl{\mathrm{in}}
\def \inlLs{\inl_{<^\sharp_{\Lex}}}
\def \inlls{\inl_{<^\sharp_{\lex}}}

\def \lex{\mathrm{lex}}
\def \Lex{\mathrm{Lex}}
\def \im{\mathrm{im}}
\def \supp{\mathrm{supp}}

\def \reg{\mathrm{reg}}
\def \NN{\mathbb N}
\def \ZZ{\mathbb Z}
\def \RR{\mathbb R}

\def \G{\mathcal G}
\def \B{\mathcal B}

\def \V{\mathcal V}

\def \A{\mathcal A}

\def \M{\mathcal M}

\begin{document}

\title{The behaviors of expansion functor on monomial ideals and toric rings}

\author{Rahim Rahmati-Asghar\footnote{Department of Mathematics, Faculty of Basic Sciences,
University of Maragheh, P. O. Box 55181-83111, Maragheh, Iran, and School of Mathematics, Institute for Research in Fundamental Sciences (IPM), P.O. Box 19395-5746, Tehran, Iran.
\url{rahmatiasghar.r@gmail.com}} and
Siamak Yassemi\footnote{School of Mathematics, Statistics and Computer Science, College of Science, University of Tehran, Tehran, Iran, and School of Mathematics, Institute for Research in Fundamental Sciences (IPM), P.O. Box 19395-5746, Tehran, Iran,
\url{yassemi@ut.ac.ir, yassemi@ipm.ir}}}

\maketitle

\abstract
In this paper we study some algebraic and combinatorial behaviors of
expansion functor. We show that on monomial ideals some properties like
polymatroidalness, weakly polymatroidalness and having linear
quotients are preserved under taking the expansion functor.

The main part of the paper is devoted to study of toric ideals
associated to the expansion of subsets of monomials which are minimal with respect
to divisibility. It is shown that,
for a given discrete polymatroid $P$, if toric ideal of $P$ is
generated by double swaps then toric ideal of any
expansion of $P$ has such a property. This result, in a special case, says that
White's conjecture is preserved under taking the expansion functor.
Finally, the construction of Gr\"{o}bner bases and some homological properties of toric ideals associated to expansions of subsets of monomials is investigated.

\section*{Introduction}

Let $S=K[x_1,\ldots,x_n]$ be a polynomial ring over a field $K$ and
let $I$ be a monomial ideal with the set of minimal generators $G(I)=\{\x^{\a_1},\ldots,\x^{\a_r}\}$ where $\x^{\a_i}=x^{\a_i(1)}_1\ldots
x^{\a_i(n)}_n$ for
$\a_i=(\a_i(1),\ldots,\a_i(n))\in\ZZ^n_+=\{\u=(u_1,\ldots,u_n)\in\ZZ^n:u_i\geq
0\}$. For the $n$-tuple $\alpha=(k_1,\ldots,k_n)\in\NN^n$, Bayati
and Herzog \cite{BaHe} defined the expansion of $I$ with respect to
$\alpha$ in the following form:

Let $S^\alpha=K[x_{11},\ldots,x_{1k_1},\ldots,x_{n1},\ldots,x_{nk_n}]$ be a polynomial ring over $K$
and set $P_j=(x_{j1},\ldots,x_{jk_j})$ a prime monomial ideal in
$S^\alpha$ for all $1\leq j\leq n$. The expansion of $I$ with
respect to $\alpha$, denoted by $I^\alpha$, is the monomial ideal
$$I^\alpha=\sum^r_{i=1}P^{\a_i(1)}_1\ldots P^{\a_i(n)}_n\subset S^\alpha$$
where $\a_i(j)$ is the $j$-th component of the vector $\a_i$.

Define the $K$-algebra homomorphism $\pi:S^\alpha\rightarrow S$ by
$\pi(x_{ij})=x_i$.

Let $\alpha=(k_1,\ldots,k_n)\in\NN^n$. For
$\u=(\u(1),\ldots,\u(n))\in\ZZ^n_+$, define $\u^\alpha$ the set of
$|\alpha|$-tuples $\w\in\ZZ^{|\alpha|}$ where $\x^\w\in
G((\x^\u)^\alpha)$. For example, if $\u=(1,2,0)$ and
$\alpha=(2,2,2)$ then
$$(\x^\u)^\alpha=(x_{11}x^2_{21}, x_{12}x^2_{21}, x_{11}x_{21}x_{22}, x_{11}x^2_{22}, x_{11}x_{21}x_{22}, x_{11}x_{21}x_{22}).$$
Therefore
\begin{center}
    $\u^\alpha=\{(1,0,2,0,0,0), (0,1,1,1,0,0), (1,0,1,1,0,0), (1,0,0,2,0,0),$\\
    $(1,0,1,1,0,0), (1,0,1,1,0,0)\}$.
\end{center}

For $\u,\v\in\ZZ^n_+$, $\u\preceq\v$ means that $\u(i)\leq\v(i)$ for all $i$. We write $\u\prec\v$ if $\u\preceq\v$ and $\u\neq\v$. If $\V$ is a set of vectors in $\ZZ^n_+$ which is minimal with respect
to $\preceq$, then $\V^\alpha=\underset{\u\in\V}{\bigcup}\u^\alpha$. Also, for a set $\A$ of monomials in $S$ which is minimal with respect to divisibility, we define $\A^\alpha=\underset{\x^\u\in\A}{\bigcup}G((\x^\u)^\alpha)$.

It is easy to see that for a monomial ideal $I\subset S$, $G(I^\alpha)=\{\x^\w:\x^\w\in G(I)^\alpha\}$.

In \cite{BaHe} the authors defined the expansion functor in the
category of finitely generated multigraded $S$-modules and studied
some homological behaviors of this functor.
In this paper, we consider a subset $\A$ of monomials which are minimal
with respect to divisibility and we study some combinatorial and homological properties on monomial ideals
generated by expansions of $\A$ and also toric ideals related to
them. Actually, we investigate some properties on $K[\A^\alpha]$ or $I_{\A^\alpha}$ when it holds for $K[\A]$ or $I_\A$. The paper is written in two main sections. One section is devoted to study of some behaviors of expansion functor on monomial ideals and the other one is on toric ideals related to expansions of subsets of monomials.

In Section 1, we show that some properties like
polymatroidalness (Theorem \ref{poly}), weakly polymatroidalness
(Theorem \ref{expan w.p.}) and having linear quotients (Theorem \ref{linear qu}) are preserved under taking the expansion functor.

In Section 2, we discuss several combinatorial and algebraic properties of expansion functor on toric algebras. White \cite{Wh} conjectured that for a matroid $\M$, the toric ideal $I_{\M}$ is generated by quadrics corresponding to double swaps. On the other hand, matroids are a special subclass of discrete polymatroids, defined in \cite{HH}. Herzog and Hibi \cite{HH} conjectured that for a discrete polymatroid $P$, the toric ideal $I_P$ is generated by quadrics corresponding to double swaps, too. We will show that for $\alpha\in\NN^n$, when the toric ideal associated to $P$ is generated by quadrics corresponding to double swaps then the toric ideal associated to $P^\alpha$ is (Theorem \ref{main}). As an application we show that if White's conjecture holds for a matroid $\M$ then it does for any expansion of $\M$. We show that the toric ring $K[\A]$ is normal and Koszul if the toric ring $K[\A^\alpha]$ is normal and Koszul. We also conclude that the reduced Gr\"{o}bner basis of $I_\A$ is the intersection of the reduced Gr\"{o}bner basis of $I_{\A^\alpha}$ with $K[\A]$. In Theorem \ref{Grob of exp}, the construction of the reduced Gr\"{o}bner basis of $I_{\A^\alpha}$ is described whenever the reduced Gr\"{o}bner basis of $I_\A$ is given. Then we show that a set of monomials is sortable if and only if its expansion is sortable(Theorem \ref{sort}). Combining this result and a result due to Sturmfels \cite{St} we conclude that if $\A$ (resp. its expansion $\A^\alpha$) is sortable then $I_{A^\alpha}$ (resp. $I_\A$) has a Gr\"{o}bner basis consisting of the quadratic sorting
relations.

We show that the toric ring of a set $\A$ of monomials is normal if and only if the toric ring of an expansion of $\A$ is (Theorem \ref{normal}). Finally, we describe some homological relations as Krull dimension, depth, projective dimension and Castelnuovo-Mumford regularity between $K[\A]$ and $K[\A^\alpha]$ (Theorem \ref{subring dim}).

\section{The expansion of some classes of monomial ideals}

In this section we show that the expansion functor has well behavior on monomial ideals
with respect to the properties of polymatroidalness, weakly
polymatroidalness and having linear quotients. In other words, we
prove that a monomial ideal has one of the mentioned properties if
and only its expansion has the same property.

\begin{defn}
(\cite{HH}) A monomial ideal $I$ of $S$ is called \emph{polymatroidal}
if it satisfies the following conditions:
\begin{enumerate}[\upshape (i)]
  \item all elements in $G(I)$ have the same degree;
  \item if $u=x^{a_1}_1\ldots x^{a_n}_n$ and $v=x^{b_1}_1\ldots x^{b_n}_n$ belong to $G(I)$ with $a_i>b_i$, then there exists $j$ with $a_j<b_j$ such that $x_j(u/x_i)\in G(I)$.
\end{enumerate}
\end{defn}

\begin{thm}\label{poly}
Let $\alpha\in\NN^n$. The monomial ideal $I\subset S$ is
polymatroidal if and only if $I^\alpha$ is polymatroidal.
\end{thm}
\begin{proof}
``Only if part'': Let $I$ be polymatroidal and let
$\alpha=(k_1,\ldots,k_n)$. Let
\begin{center}
$u=x^{a_{11}}_{11}\ldots x^{a_{1k_1}}_{1k_1}\ldots x^{a_{n1}}_{n1}\ldots x^{a_{nk_n}}_{nk_n}$
and $v=x^{b_{11}}_{11}\ldots x^{b_{1k_1}}_{1k_1}\ldots
x^{b_{n1}}_{n1}\ldots x^{b_{nk_n}}_{nk_n}$
\end{center}
be two monomials in
$G(I^\alpha)$ with $a_{ij}>b_{ij}$. Set $a_i=\sum_ja_{ij}$ and
$b_i=\sum_jb_{ij}$. If $a_i\leq b_i$ then there exists some $j'$
such that $a_{ij'}<b_{ij'}$. Therefore it is clear that $x_{ij'}(u/x_{ij})\in
G(I^\alpha)$ and the assertion holds. So suppose that
$a_i>b_i$. Then there exists $k$ with $a_k<b_k$ such that
$x_k(\pi(u)/x_i)\in G(I)$. In particular, $a_{k}<b_k$ implies that there is $t$
with $a_{kt}<b_{kt}$. Therefore $x_{kt}(u/x_{ij})\in G(I^\alpha)$
which is desired assertion.

``If part'': Let $I^\alpha$ be polymatroidal and let
$u=x^{a_1}_1\ldots x^{a_n}_n$ and $v=x^{b_1}_1\ldots x^{b_n}_n$ with
$a_i>b_i$ be two monomials of $G(I)$. Then $u'=x^{a_1}_{11}\ldots
x^{a_n}_{n1}$ and $v'=x^{b_1}_{11}\ldots x^{b_n}_{n1}$ are two
monomials in $G(I^\alpha)$. Hence there is $j$ with $a_j<b_j$ such
that $x_{j1}(u'/x_{i1})\in G(I^\alpha)$. Therefore $x_j(u/x_i)\in G(I)$.
\end{proof}

\begin{defn}
(\cite{HK,MM}) A monomial ideal $I$ is called \emph{weakly
polymatroidal with respect to the ordering} $x_1>\ldots>x_n$ if for every two monomials $u=x^{a_1}_1\ldots
x^{a_n}_n$ and $v=x^{b_1}_1\ldots x^{b_n}_n$ in $G(I)$ such that
$a_1=b_1,\ldots,a_{t-1}=b_{t-1}$ and $a_t>b_t$, there exists $j>t$
such that $x_t(v/x_j)\in I$.

We say that a monomial ideal $I\subset S$ is weakly polymatroidal if it is weakly polymatroidal with respect to some ordering of variables $x_1,\ldots,x_n$.
\end{defn}

It is clear from the definition that a polymatroidal ideal is weakly polymatroidal but the converse is not true in general(see \cite[Example 1.3]{HK}).

\begin{thm}\label{expan w.p.}
Let $\alpha\in\NN^n$. The monomial ideal $I\subset S$ is weakly
polymatroidal if and only if $I^\alpha\subset S^\alpha$ is weakly
polymatroidal.
\end{thm}
\begin{proof}
``Only if part'': Suppose that $I$ is weakly polymatroidal with
respect to the ordering $x_1>\ldots>x_n$. Let
$\alpha=(k_1,\ldots,k_n)$. We want to show that $I^\alpha$ is weakly
polymatroidal with respect to the ordering
$x_{11}>\ldots>x_{1k_1}>\ldots>x_{n1}>\ldots>x_{nk_n}.$
Let $u=x^{a_{11}}_{11}\ldots x^{a_{1k_1}}_{1k_1}\ldots x^{a_{n1}}_{n1}\ldots x^{a_{nk_n}}_{nk_n}$ and $v=x^{b_{11}}_{11}\ldots x^{b_{1k_1}}_{1k_1}\ldots
x^{b_{n1}}_{n1}\ldots x^{b_{nk_n}}_{nk_n}$ be two monomials in $G(I^\alpha)$ such that one of the following properties holds:

\begin{enumerate}[\upshape (i)]
  \item $a_{11}=b_{11},\ldots,a_{s-1\ k_{s-1}}=b_{s-1\ k_{s-1}}$ and $a_{s1}>b_{s1}$.
  \item $a_{11}=b_{11},\ldots,a_{s\ t-1}=b_{s\ t-1}$ and $a_{st}>b_{st}$.
\end{enumerate}

Set $a_i=\sum_ja_{ij}$ and $b_i=\sum_jb_{ij}$. Suppose (i) holds. We
have two cases:

Case 1. Let $x_s|\pi(v)$. If $k_s=1$, then there is $t>s$ such that
$x_s(\pi(v)/x_t)\in I$. So $x_{s1}(v/x_{tl})\in I^\alpha$ for some
$l$. So assume that $k_s>1$. If $b_s=1$, then there is $t'>s$ such that
$x_s(\pi(v)/x_{t'})\in I$ and so $x_{s1}(v/x_{t'l})\in I^\alpha$ for
some $l$. If $b_s>1$, then by the definition of expansion of an
ideal there is $p>1$ such that $x_{sp}|v$ and clearly $x_{s1}(v/x_{sp})\in
I^\alpha$.

Case 2. Let $x_s\nmid \pi(v)$. Since $I$ is weakly polymatroidal,
there is $k>s$ such that $x_s(\pi(v)/x_k)\in I$. Thus
$x_{s1}(v/x_{kl})\in I^\alpha$ for some $l$.

Suppose (ii) holds. If there is $t'>t$ such that $x_{st'}|v$, then
it is clear that $x_{st}(v/x_{st'})\in I^\alpha$. Assume $x_{st'}\nmid v$, for all
$t'>t$. Then $a_s>b_s$ and since $I$ is weakly
polymatroidal we have $x_{s}(\pi(v)/x_k)\in I$ for some $k>s$. This
implies that $x_{st}(v/x_{kl})\in I^\alpha$ for some $l$.

Therefore $I^\alpha$ is weakly polymatroidal.

``If part'': Suppose $I^\alpha$ is weakly polymatroidal with respect to
the ordering
\begin{equation}\label{1}
x_{i_1j_1}>\ldots>x_{i_{|\alpha|}j_{|\alpha|}}
\end{equation}
which for all $l$, $1\leq i_l\leq n$ and $1\leq j_l\leq k_{i_l}$. We will show that $I$ is weakly polymatroidal with respect to the
ordering $x_{s_1}>\ldots>x_{s_n}$ obtained from the ordering (\ref{1})
after applying the $K$-algebra homomorphism $\pi:S^\alpha\rightarrow S$ by $\pi(x_{ij})=x_i$ and removing the repeated variables
beginning on the left-hand. In other words, if
$$x_{i_1}\geq\ldots\geq x_{i_p}\geq\ldots\geq x_{i_q}\geq\ldots\geq x_{i_{|\alpha|}}$$
where $i_p=i_q$ then we will remove $x_{i_q}$. Let $x_{s_l}=\pi(x_{s_lt_l})$ for all
$l=1,\ldots,n$. Let $u=x^{a_1}_{s_1}\ldots
x^{a_n}_{s_n},v=x^{b_1}_{s_1}\ldots x^{b_n}_{s_n}\in G(I)$ with
$a_1=b_1,\ldots,a_{j-1}=b_{j-1}$ and $a_j>b_j$. Then
$u'=x^{a_1}_{s_1t_1}\ldots
x^{a_n}_{s_nt_n}>v'=x^{a_1}_{s_1t_1}\ldots x^{a_n}_{s_nt_n}$ are in
$G(I^\alpha)$ and so there exists $k>j$ with $a_k<b_k$ such that
$x_{s_jt_j}(v'/x_{s_kt_k})\in I^\alpha$. Thus $x_{s_j}(v/x_{s_k})\in
I$. This concludes that $I$ is weakly polymatroidal.
\end{proof}

Now we study the behavior of expansion functor on the property of having linear
quotients. First, we recall some notations and definitions from \cite{RaYa}:

For the monomial $u=x^{a_1}_1\ldots x^{a_n}_n$ in $S$, we will denote the \emph{support} of $u$ by
$\supp(u)$ and it is the set of those integers $i$ that $a_i\neq 0$. Set
$\nu_{x_i}(u):=a_i$. When
$M$ is another monomial, we set $[u,M]=1$ if for all
$i\in\supp(u)$, $x^{a_i}_i\nmid v$. Otherwise we set $[u,M]\neq
1$.

For the monomial $u\in S$ and the monomial ideal $I\subset S$ set
\begin{center}
$I^u=\langle M_i\in G(I): [u,M_i]\neq 1\rangle$\quad and\quad
$I_u=\langle M_i\in G(I): [u,M_i]=1\rangle.$
\end{center}

Let the minimal system of generators of $I$ be
$G(I)=\{M_1,\ldots,M_r\}$. The monomial $u=x^{a_1}_1\ldots
x^{a_n}_n$ is called \emph{shedding} if $I_u\neq 0$ and for each
$M_i\in G(I_u)$ and each $l\in\supp(u)$ there exists $M_j\in G(I^u)$
such that $M_j:M_i=x_l$.

\begin{defn}
(\cite{RaYa}) Let $I$ be a monomial ideal minimally generated by
$\{M_1,\ldots,M_r\}$. We say $I$ is a \emph{$k$-decomposable} ideal
if $r=1$ or else has a shedding monomial $u$ with $|\supp(u)|\leq
k+1$ such that the ideals $I^u$ and $I_u$ are $k$-decomposable.
\end{defn}

For each monomial ideal $I$ and a system of minimal generators
$u_1,\ldots,u_r$ of $I$, we say that $I$ has \emph{linear quotients
with respect to the ordering} $u_1,\ldots,u_r$ if for all
$j=2,\ldots,r$, the colon ideal $(u_1,\dots,u_{j-1}):(u_j)$ is
generated by linear forms. In other words, if for all $j<i$ there exists an integer $k<i$ and an integer $l$ such
that
$$\frac{u_k}{\gcd(u_k,u_i)}=x_l|\frac{u_j}{\gcd(u_j,u_i)}.$$
A monomial ideal which has linear quotients with respect to some ordering of minimal generators, is called a monomial ideal with linear quotients.

It is known that a weakly polymatroidal ideal has linear quotients. In \cite{RaYa} the authors proved that

\begin{thm}\label{decompos}
(\cite[Theorem 2.13]{RaYa}) A monomial ideal has linear quotients if and only if it is $k$-decomposable for some
$k\geq 0$.
\end{thm}

We will use this theorem in the following.

\begin{thm}\label{linear qu}
Let $\alpha\in\NN^n$. The monomial ideal $I\subset S$ has linear
quotients if and only if $I^\alpha\subset S^\alpha$ has linear
quotients.
\end{thm}
\begin{proof}
``Only if part'': We use induction on the
number of minimal generators of $I$. Since $I$ has linear
quotients, so there exists a shedding monomial
$u=x^{a_1}_{i_1}\ldots x^{a_t}_{i_m}$ for $I$ such that $I^u$ and $I_u$ have linear quotients, by Theorem \ref{decompos}. Set $J=I^\alpha$ and
\begin{center}
$J^0=(M\in G(J)|[u,\pi(M)]\neq 1)$ and $J_0=(M\in G(J)|[u,\pi(M)]=1).$
\end{center}
Note that $J^0=(I^u)^\alpha$ and $J_0=(I_u)^\alpha$. Let $J^0$ and $J_0$ are, respectively, minimally generated by
$M_1,\ldots,M_r$ and $M_{r+1},\ldots,M_s$ and moreover, they have linear quotients with respect the given orderings. We want to show that $J$ has linear quotients with respect to
$M_1,\ldots,M_s$.

Let $M_p$ and $M_q$ are in $G(J)$ with $p<q$. Set $N_i:=\pi(M_i)$. By the induction
hypothesis, $J^0$ and $J_0$ have linear quotients. Thus it suffices to consider that $p\leq r<q$. It is clear
that $N_p\in G(I^u)$ and $N_q\in G(I_u)$. Therefore there is $t\leq
p$ such that $N_t/\gcd(N_t,N_q)=x_{i_l}$ for some $1\leq l\leq m$
and $x_{i_l}$ divides $N_p/\gcd(N_p,N_q)$. It concludes that $M_p/\gcd(M_p,M_q)$ is divided by $x_{i_lh}$ for some $1\leq h\leq k_{i_l}$. To complete the assertion, one can choose a monomial $M_{t'}\in\pi^{-1}(N_t)$ with the property $M_{t'}/\gcd(M_{t'},M_q)=x_{i_lh}$. Clearly, $t'\leq p$.

``If part'': Let $I^\alpha$ has linear quotients with respect to the
ordering $M_1,\ldots,M_s$. Set
$N_l:=\pi(M_l)$. Suppose that we obtain the ordering
$N_{i_1},\ldots,N_{i_r}$, with disjoint monomials, after removing
any repeated monomial beginning on the left-hand of the ordering $N_1,\ldots,N_s$. We want to show that $I$ has linear
quotients with respect to $N_{i_1},\ldots,N_{i_r}$.


Consider two monomials $N_{i_p}$ and $N_{i_q}$ with $p<q$. Let
$N_{i_p}=\pi(M_{p'})$ and
$N_{i_q}=\pi(M_{q'})$. Clearly, $p'<q'$. Therefore there exist $k'<q'$ and a
variable $x_{lm}$ such that $M_{k'}/\gcd(M_{k'},M_{q'})=x_{lm}$ and, moreover, $x_{lm}$ divides $M_{p'}/\gcd(M_{p'},M_{q'})$. Let $N_{i_k}=\pi(M_{k'})$. Since that $N_{i_k}\nmid N_{i_q}$, we have $\gcd(N_{i_k},N_{i_q})\neq N_{i_k}$. Therefore $N_{i_k}/\gcd(N_{i_k},N_{i_q})\neq 1$. Now let $x^a_t$ divide $N_{i_k}/\gcd(N_{i_k},N_{i_q})$. Then for some $s$, $x_{ts}$ divides $M_{k'}/\gcd(M_{k'},M_{q'})$. This implies that $a=1$, $t=l$ and $s=m$. Therefore $N_{i_k}/\gcd(N_{i_k},N_{i_q})=x_l$. Similarly, we show that $x_l$ divides $N_{i_p}/\gcd(N_{i_p},N_{i_q})$, as desired.
\end{proof}

\begin{rem}
The only if part of Theorem \ref{linear qu} was proved, in a different argument, in Proposition 1.6 of \cite{BaHe}.
\end{rem}

\begin{rem}
For a given $\alpha\in\NN^n$ and a monomial ideal $I\subset S$ generated in one degree, we can abbreviate Theorems \ref{poly}, \ref{expan w.p.}, \ref{linear qu}
and also Theorem 4.2 of \cite{BaHe} in the following implications:
$$\begin{array}{ccc}
    I\ \mbox{is polymatroidal} & \Leftrightarrow & I^\alpha\ \mbox{is polymatroidal} \\
    \Downarrow &  & \Downarrow \\
    I\ \mbox{is weakly polymatroidal} & \Leftrightarrow & I^\alpha\ \mbox{is weakly polymatroidal} \\
    \Downarrow &  & \Downarrow \\
    I\ \mbox{has linear quotients} & \Leftrightarrow & I^\alpha\ \mbox{has linear quotients} \\
    \Downarrow &  & \Downarrow \\
    I\ \mbox{has linear resolution} & \Leftrightarrow & I^\alpha\ \mbox{has linear resolution}
  \end{array}
$$
\end{rem}

\section{The expansion functor and toric algebra}

\subsection{White's conjecture}

White \cite{Wh} conjectured that for a matroid $\M$, the toric ideal
associated to $\M$, $I_\M$, is generated by quadrics corresponding to
double swaps. In the previous section it was shown that a monomial
ideal is polymatroidal if and only if its expansion is
polymatroidal. Since that a polymatroidal ideal is generated by
monomials corresponding to the base of a discrete polymatroid, and
also since matroids are a special subclass of discrete polymatroids,
it is then natural to ask about holding white's conjecture for
expansion of a discrete polymatroid. We first bring some notations and definitions.

\begin{defn}
(\cite{HH}) A \emph{discrete polymatroid} on the ground set $[n]$ is a nonempty finite set $P\subset\ZZ^n_+$ satisfying

(D1) if $\v\in\ZZ^n_+$ with $\v\preceq\u$ for some $\u\in P$, then $\v\in P$;

(D2) if $\u,\v\in P$ with $|\u|<|\v|$, then there is $i\in [n]$ with $\u(i)<\v(i)$ such that $\u+\epsilon_i\in P$. Here $\epsilon_i$ denotes the $i$th canonical basis vector in $\RR^n$.
\end{defn}

A \emph{base} of $P$ is a vector $\u\in P$ such that $\u\prec\v$ for no $\v\in P$. It follows from (D1) and (D2) that a nonempty finite set $B\subset\ZZ^n_+$ is the set of bases of a discrete polymatroid on $[n]$ if and only if $B$ satisfies the following conditions:

\begin{enumerate}[\upshape (i)]
  \item all elements of $B$ have the same modulus;
  \item if $\u,\v\in P$ belong to $B$ with $\u(i)>\v(i)$, then there is $j\in [n]$ with $\u_j<\v_j$ such that $\u-\epsilon_i+\epsilon_j\in B$.
\end{enumerate}
We will denote by $\B_P$ the set of bases of $P$ on $[n]$.

Let $P\subset \ZZ^n_+$ be a discrete polymatroid and $\B_P$ its set
of bases. Define $S_P=K[y_\u:\u\in\B_P]$ a polynomial ring over $K$
and write $I_P\subset S_P$ for the toric ideal of the \emph{base
ring} $K[P]:=K[\x^\u:\u\in \B_P]$ where $\x^\u=x^{\u(1)}_1\ldots
x^{\u(n)}_n$ for $\u=(\u(1),\ldots,\u(n))\in\ZZ^n_+$. In other
words, $I_P$ is the kernel of the surjective $K$-algebra
homomorphism $\varphi_P:S_P\rightarrow K[P]$ defined by
$\varphi_P(y_\u)=\x^\u$.

We say that a pair of bases $(\v_1,\v_2)$ is obtained from a pair of
bases $(\u_1,\u_2)$ by a \emph{double swap} if
$\v_1=\u_1+\epsilon_j-\epsilon_i$ and
$\v_2=\u_2+\epsilon_i-\epsilon_j$ for some $i,j$ with $\u_1(i)>\u_2(i)$ and
$\u_2(j)>\u_1(j)$. In this case we write $(\v_1,\v_2)\sim_P (\u_1,\u_2)$. 

Recall that for the canonical basis vector $\epsilon_i\in\RR^n$,
$$\epsilon^\alpha_i=\{\epsilon_{i1},\ldots,\epsilon_{ik_i}\}$$ which
$\epsilon_{ij}$ is a canonical basis vector of $\RR^{|\alpha|}$ with
the $(k_1+\ldots+k_{i-1}+j)$-th component equal to 1 and zero for
other components.

For a discrete polymatroid $P\subset\ZZ^n_+$, $P^\alpha\subset\ZZ^{|\alpha|}$ is defined a discrete
polymatroid which its set of bases is
$\B_{P^\alpha}=\underset{\u\in\B_P}{\bigcup}\u^\alpha$.

\begin{exam}
Consider the discrete polymatroid $P$ with the singleton base set
$\B_P=\{\u:=(1,1)\}$ and let $\alpha=(2,2)$. Then
$$\B_{P^\alpha}=\{\u_1:=(1,0,1,0),\u_2:=(1,0,0,1),\u_3:=(0,1,1,0),\u_4:=(0,1,0,1)\}.$$
Moreover,
\begin{center}
$I_P=0$ and
$I_{P^\alpha}=(y_{\u_1}y_{\u_4}-y_{\u_2}y_{\u_3})$.
\end{center}
\end{exam}

Define the surjective map $$\pi_0:\ZZ^{|\alpha|}\rightarrow\ZZ^n$$ by
$\pi_0(\u)=(a_1,\ldots,a_n)$ for
$\u=(a_{11},\ldots,a_{1k_1},\ldots,a_{n1},\ldots,a_{nk_n})\in\ZZ^{|\alpha|}$
where $a_j=\sum^{k_j}_{l=1}a_{jl}$ for all $j$. We also define the
$K$-algebra epimorphism
$$\begin{array}{rcl}
    \tau:K[P^\alpha]&\rightarrow & K[P] \\
    \x^\u & \mapsto & \x^{\pi_0(\u)}.
  \end{array}
$$

Also, the $K$-algebra homomorphism $\gamma:S_{P^\alpha}\rightarrow S_P$ is
defined by $\gamma(y_{\u})=y_{\pi_0(\u)}$. Therefore we have the
following commutative diagram from surjective maps:
$$\begin{array}{ccc}
  S_{P^\alpha} & \overset{\varphi_{P^\alpha}}{\longrightarrow} & K[P^\alpha] \\
  \gamma \downarrow &  & \downarrow\tau \\
  S_P & \overset{\varphi_P}{\longrightarrow} & K[P].
\end{array}$$

Before proving the main theorem of this subsection we require the following
lemmas.

\begin{lem}\label{pi0}
For $\alpha\in\NN^n$, $\gamma(I_{P^\alpha})=I_P$.
\end{lem}
\begin{proof}
Let $f\in I_{P^\alpha}$. Then
$\varphi_P(\gamma(f))=\tau(\varphi_{P^\alpha}(f))=0$ and so
$\gamma(f)\in I_P$. For the converse inclusion, let
$g=y_{\u_1}\ldots y_{\u_m}-y_{\v_1}\ldots y_{\v_m}\in I_P$. Then for all $i$, set $\u'_i:=\sum^m_{j}\u_i(j)\epsilon_{j1}$,
$\v'_i:=\sum^m_{j}\v_i(j)\epsilon_{j1}$ and $h:=y_{\u'_1}\ldots
y_{\u'_m}-y_{\v'_1}\ldots y_{\v'_m}$. Hence $\gamma(h)=g$. By the
fact that a binomial $y_{\u_1}\ldots y_{\u_m}-y_{\v_1}\ldots
y_{\v_m}\in I_P$ if and only if $\sum^m \u_i=\sum^m \v_i$, we
conclude that $\sum^m \u'_i=\sum^m \v'_i$. Therefore $h\in
I_{P^\alpha}$ and so $g\in\gamma(I_{P^\alpha})$.
\end{proof}



\begin{thm}\label{main}
Let $\alpha\in\NN^n$ and $P$ be a discrete polymatroid. If $I_P$ is
generated by quadratic binomials corresponding to double swaps then $I_{P^\alpha}$ is, too.
\end{thm}
\begin{proof}
Suppose $I_P$ is
generated by quadratic binomials corresponding to double swaps. We consider a binomial $f\in I_{P^\alpha}$ and show that we can write $f$ as sum of quadratic binomials corresponding to double swaps.

Suppose $0\neq f\in
I_{P^\alpha}$ is of degree 2. Let $f=y_{\u'_1}y_{\u'_2}-y_{\v'_1}y_{\v'_2}$. Let $\pi_0(\u'_i)=\u_i$ and $\pi_0(\v'_i)=\v_i$. By
Lemma \ref{pi0},
$\gamma(y_{\u'_1}y_{\u'_2}-y_{\v'_1}y_{\v'_2})=y_{\u_1}y_{\u_2}-y_{\v_1}y_{\v_2}\in
I_P$. If $y_{\u_1}y_{\u_2}-y_{\v_1}y_{\v_2}=0$, then one can assume that $\u_1=\v_1$ and $\u_2=\v_2$. It follows from $\u'_1+\u'_2=\v'_1+\v'_2$ that
$$\v'_1=\u'_1+\w'-\z',\ \v'_2=\u'_2+\z'-\w'$$
where $\w'=\sum^{f_r}_{i=f_1}\sum^{g_s}_{j=g_1} t_{ij}\epsilon_{ij}$, $\z'=\sum^{f_r}_{i=f_1}\sum^{h_t}_{j=h_1}s_{ij}\epsilon_{ij}$ and, moreover, $\sum_j t_{ij}=\sum_j s_{ij}$. Therefore we can write $y_{\u'_1}y_{\u'_2}-y_{\v'_1}y_{\v'_2}$ as a sum of quadratic binomials corresponding to double swaps:
$$\begin{array}{rl}
  y_{\u'_1}y_{\u'_2}-y_{\v'_1}y_{\v'_2}= & (y_{\u'_1}y_{\u'_2}-y_{\u'_1+\epsilon_{f_1g_1}-\epsilon_{f_1h_1}}y_{\u'_2+\epsilon_{f_1h_1}-\epsilon_{f_1g_1}})+\ldots+ \\
   & (y_{\u'_1+\sum^{f_r}_{i=f_1}\sum^{g_{s-1}}_{j=g_1} t_{ij}\epsilon_{ij}-\sum^{f_r}_{i=f_1}\sum^{h_{t-1}}_{j=h_1}s_{ij}\epsilon_{ij}}y_{\u'_2+\sum^{f_r}_{i=f_1}\sum^{h_{t-1}}_{j=h_1}s_{ij}\epsilon_{ij}-\sum^{f_r}_{i=f_1}\sum^{g_{s-1}}_{j=g_1} t_{ij}\epsilon_{ij}}-
y_{\v'_1}y_{\v'_2}).
\end{array}$$

Thus suppose that $y_{\u_1}y_{\u_2}-y_{\v_1}y_{\v_2}\neq 0$.
By the assumption, we have
$$y_{\u_1}y_{\u_2}-y_{\v_1}y_{\v_2}=\sum^t_{i=1}(y_{\u_{i1}}y_{\u_{i2}}-y_{\v_{i1}}y_{\v_{i2}})$$
where $(\v_{i1},\v_{i2})\sim_P (\u_{i1},\u_{i2})$ for all $i$.

\textbf{Case 1:} If $t=1$, then $(\v_1,\v_2)\sim_P (\u_1,\u_2)$. Let
$\v_1=\u_1+\epsilon_{i_q}-\epsilon_{i_p}$ and
$\v_2=\u_2+\epsilon_{i_p}-\epsilon_{i_q}$.

Then we will have
$\v'_1=\u'_1+\epsilon_{i_qj_{q'}}-\epsilon_{i_pj_{p'}}$ and
$\v'_2=\u'_2+\epsilon_{i_pj_{p''}}-\epsilon_{i_qj_{q''}}$. Now $\u'_1+\u'_2=\v'_1+\v'_2$ implies that $j_{p'}=j_{p''}$ and
$j_{q'}=j_{q''}$. Thus $(\v'_1,\v'_2)\sim_{P^\alpha} (\u'_1,\u'_2)$.

\textbf{Case 2:} If $t>1$, then by choosing suitable $u'_{ij}$'s from $\B_{P^\alpha}$ one may consider
$$y_{\u'_1}y_{\u'_2}-y_{\v'_1}y_{\v'_2}=\sum^t_{i=1}(y_{\u'_{i1}}y_{\u'_{i2}}-y_{\v'_{i1}}y_{\v'_{i2}})$$
where $\pi_0(\u'_{ij})=\u_{ij}$ and $\pi_0(\v'_{ij})=\v_{ij}$ for
all $i$ and $j$ and, moreover, $y_{\u'_{i1}}y_{\u'_{i2}}-y_{\v'_{i1}}y_{\v'_{i2}}\in I_{P^\alpha}$.
It follows from case 1 that $y_{\u'_1}y_{\u'_2}-y_{\v'_1}y_{\v'_2}$ is generated by quadratic
binomials corresponding to double swaps.

Now suppose that $f=y_{\u'_1}\ldots y_{\u'_m}-y_{\v'_1}\ldots
y_{\v'_m}\in I_{P^\alpha}$ where $m>2$ and every binomial
$y_{\u'_1}\ldots y_{\u'_{m'}}-y_{\v'_1}\ldots y_{\v'_{m'}}\in
I_{P^\alpha}$ with $m'<m$ is generated by quadratic binomials
corresponding to double swaps. If $\u'_i=\v'_j$ for some $1\leq
i,j\leq m$ then
$$y_{\u'_1}\ldots y_{\u'_m}-y_{\v'_1}\ldots y_{\v'_m}=y_{\u'_i}(y_{\u'_1}\ldots \hat{y}_{\u'_i}\ldots y_{\u'_m}-y_{\v'_1}\ldots \hat{y}_{\v'_j}\ldots y_{\v'_m})$$
which is generated by quadratic binomials corresponding to double
swaps, by induction hypothesis. So assume that $\u'_i\neq \v'_j$ for
all $1\leq i,j\leq m$. Let $\u_i=\pi_0(\u'_i)$ and $\v_i=\pi_0(\v'_i)$, for all $i$. It
follows from $\sum^m \u_i=\sum^m \v_i$ that $y_{\u_1}\ldots
y_{\u_m}-y_{\v_1}\ldots y_{\v_m}\in I_P$. First, let $y_{\u_1}\ldots
y_{\u_m}-y_{\v_1}\ldots y_{\v_m}=0$. Then we can suppose that $\u_1=\v_1$. Let $\u'_1(ij)>\v'_1(ij)$ and $\u'_1(ik)<\v'_1(ik)$. This implies that there is $\u'_l$ with $\u'_l(ik)>0$. For convenience, let us set $l=2$. Hence we can write
$$\begin{array}{rl}
    y_{\u'_1}\ldots y_{\u'_m}-y_{\v'_1}\ldots y_{\v'_m}= & (y_{\u'_1}y_{\u'_2}\ldots y_{\u'_m}-y_{\u'_1-\epsilon_{ij}+\epsilon_{ik}}y_{\u'_2-\epsilon_{ik}+\epsilon_{ij}}y_{\u'_3}\ldots y_{\u'_m})+ \\
     & (y_{\u'_1-\epsilon_{ij}+\epsilon_{ik}}y_{\u'_2-\epsilon_{ik}+\epsilon_{ij}}y_{\u'_3}\ldots y_{\u'_m}
-y_{\v'_1}\ldots y_{\v'_m})\\
     =& (y_{\u'_1}y_{\u'_2}-y_{\u'_1-\epsilon_{ij}+\epsilon_{ik}}y_{\u'_2-\epsilon_{rs}+\epsilon_{ij}})y_{\u'_3}\ldots y_{\u'_m}+ \\
     & (y_{\u'_1-\epsilon_{ij}+\epsilon_{ik}}y_{\u'_2-\epsilon_{ik}+\epsilon_{ij}}y_{\u'_3}\ldots y_{\u'_m}
-y_{\v'_1}\ldots y_{\v'_m}).
  \end{array}
$$
Set $\u''_1=\u'_1-\epsilon_{ij}+\epsilon_{ik}$ and $\u''_2=\u'_2-\epsilon_{ik}+\epsilon_{ij}$. Clearly, $y_{\u''_1} y_{\u''_2}y_{\u'_3}\ldots y_{\u'_m}-y_{\v'_1}\ldots
y_{\v'_m}\in I_{P^\alpha}$ and $\pi_0(\u''_1)=\pi_0(\v'_1)$. Now we repeat the above procedure for $y_{\u''_1} y_{\u''_2}y_{\u'_3}\ldots y_{\u'_m}-y_{\v'_1}\ldots
y_{\v'_m}$ and after a finite number of steps, we obtain
$$y_{\u'_1}\ldots y_{\u'_m}-y_{\v'_1}\ldots y_{\v'_m}=\sum f_ig_i+y_{\v'_1}(y_{\u'''_2}y_{\u'''_3}\ldots y_{\u'''_m}
-y_{\v'_2}\ldots y_{\v'_m})$$
where $f_i$'s are monomials in $S_{P^\alpha}$ and $g_i$'s are quadratic binomials corresponding to double
swaps. By induction hypothesis, $y_{\u'''_2}y_{\u'''_3}\ldots y_{\u'''_m}-y_{\v'_2}\ldots y_{\v'_m}\in I_{P^\alpha}$ is generated by quadratic binomials corresponding to double swaps. Therefore the assertion holds when $y_{\u_1}\ldots
y_{\u_m}-y_{\v_1}\ldots y_{\v_m}=0$. So assume that $y_{\u_1}\ldots
y_{\u_m}-y_{\v_1}\ldots y_{\v_m}\neq 0$. Therefore, by assumption,
$$y_{\u_1}\ldots y_{\u_m}-y_{\v_1}\ldots y_{\v_m}=\sum^t_{i=1}f_i(y_{\u_{i1}}y_{\u_{i2}}-y_{\v_{i1}}y_{\v_{i2}})$$
where $(\v_{i1},\v_{i2})\sim_P (\u_{i1},\u_{i2})$ and $f_i$'s are monomials in $S_P$. Without loss of
generality we may assume that
\begin{equation}\label{2}
y_{\u_1}\ldots y_{\u_m}-y_{\v_1}\ldots y_{\v_m}=\sum^t_{i=1}(y_{\u_{i1}}y_{\u_{i2}}y_{\u_{i3}}\ldots y_{\u_{im}}-y_{\v_{i1}}y_{\v_{i2}}y_{\u_{i3}}\ldots y_{\u_{im}})
\end{equation}
which for every $i$, $(\v_{i1},\v_{i2})\sim_P(\u_{i1},\u_{i2})$. Clearly, for all $i$, $y_{\u_{i1}}y_{\u_{i2}}y_{\u_{i3}}\ldots y_{\u_{im}}-y_{\v_{i1}}y_{\v_{i2}}y_{\u_{i3}}\ldots y_{\u_{im}}\in I_P$.

\textbf{Case $1'$:} If $t=1$, then we may assume that $\u_{1j}=\u_j$ for all $j=1,\ldots,m$, $\v_{1j}=\v_j$ for $j=1,2$ and $\v_{1j}=\u_j$
for all $j=3,\ldots,m$. Let $\v_1=\u_1-\varepsilon_i+\varepsilon_j$ and $\v_2=\u_2-\varepsilon_j+\varepsilon_i$. Since that $\u_1(i)>\u_2(i)$ and $\u_1(j)<\u_2(j)$, it follows that there are $r$ and $s$ such that $\u'_1(ir)>\u'_2(ir)$ and $\u'_1(js)<\u'_2(js)$. We have
$$\begin{array}{rl}
    y_{\u'_1}\ldots y_{\u'_m}-y_{\v'_1}\ldots y_{\v'_m}= & (y_{\u'_1}y_{\u'_2}-y_{\u'_1-\varepsilon_{ir}+\varepsilon_{js}}y_{\u'_2-\varepsilon_{js}+\varepsilon_{ir}})y_{\u'_3}\ldots y_{\u'_m} \\
     &+ y_{\u'_1-\varepsilon_{ir}+\varepsilon_{js}}y_{\u'_2-\varepsilon_{js}+\varepsilon_{ir}}y_{\u'_3}\ldots y_{\u'_m}-y_{\v'_1}\ldots y_{\v'_m}
  \end{array}
$$
Note that $\pi_0(\u'_1-\varepsilon_{ir}+\varepsilon_{js})=\pi_0(\v'_1)$ and $\pi_0(\u'_2-\varepsilon_{js}+\varepsilon_{ir})=\pi_0(\v'_2)$ and $y_{\u'_1-\varepsilon_{ir}+\varepsilon_{js}}y_{\u'_2-\varepsilon_{js}+\varepsilon_{ir}}y_{\u'_3}\ldots y_{\u'_m}-y_{\v'_1}\ldots y_{\v'_m}\in I_{P^\alpha}$. Since $y_{\pi_0(\u'_1-\varepsilon_{ir}+\varepsilon_{js})}y_{\pi_0(\u'_2-\varepsilon_{js}+\varepsilon_{ir})}y_{\pi_0(\u'_3)}\ldots y_{\pi_0(\u'_m)}-y_{\pi_0(\v'_1)}\ldots y_{\pi_0(\v'_m)}=0$, it follows from case 2 that $y_{\u'_1-\varepsilon_{ir}+\varepsilon_{js}}y_{\u'_2-\varepsilon_{js}+\varepsilon_{ir}}y_{\u'_3}\ldots y_{\u'_m}-y_{\v'_1}\ldots y_{\v'_m}$ is generated by quadratic binomials corresponding to double swaps and so $y_{\u'_1}\ldots y_{\u'_m}-y_{\v'_1}\ldots y_{\v'_m}$ is generated by quadratic binomials corresponding to double swaps.


\textbf{Case $2'$:} If $t>1$, then considering the equality (\ref{2}), we can choose some suitable bases $\u'_{ij}$
and $\v'_{ij}$ of $\B_{P^\alpha}$, with $\pi_0(\u'_{ij})=\u_{ij}$
and $\pi_0(\v'_{ij})=\v_{ij}$ such that
$$y_{\u'_1}\ldots y_{\u'_m}-y_{\v'_1}\ldots y_{\v'_m}=\sum^t_{i=1}(y_{\u'_{i1}}y_{\u'_{i2}}y_{\u'_{i3}}\ldots y_{\u'_{im}}-y_{\v'_{i1}}y_{\v'_{i2}}y_{\v'_{i3}}\ldots y_{\v'_{im}})$$
and $y_{\u'_{i1}}y_{\u'_{i2}}y_{\u'_{i3}}\ldots y_{\u'_{im}}-y_{\v'_{i1}}y_{\v'_{i2}}y_{\v'_{i3}}\ldots y_{\v'_{im}}\in I_{P^\alpha}$.
Now by the case $1'$, every binomial $y_{\u'_{i1}}y_{\u'_{i2}}y_{\u'_{i3}}\ldots
y_{\u'_{im}}-y_{\v'_{i1}}y_{\v'_{i2}}y_{\v'_{i3}}\ldots
y_{\v'_{im}}$ is generated by quadratic binomials corresponding to
double swaps and so $y_{\u'_1}\ldots y_{\u'_m}-y_{\v'_1}\ldots
y_{\v'_m}$ has the same property, as desired.

\end{proof}

Since a matroid may be regarded as a discrete polymatroid with
the set of bases consisting of $(0,1)$-vectors, so we can conclude
the following.

\begin{cor}\label{white}
Let $\alpha\in\NN^n$ and $\M$ be a matroid. If $\M$ satisfies the White's
conjecture then $\M^\alpha$ does, too.
\end{cor}


\begin{rem}
We conjecture that the converse of Theorem \ref{main} holds but we could not present any proof. Indeed, this will conclude the converse of Corollary \ref{white}.
\end{rem}

\subsection{Some combinatorial and algebraic properties}

Let $\A=\{u_1,\ldots,u_m\}$ be a set of monomials belonging to
$S=K[x_1,\ldots,x_n]$ and suppose that the affine semigroup ring
$K[\A]=K[u_1,\ldots,u_m]$ is a homogeneous $K$-algebra. Let
$S_\A:=K[y_{u_1},\ldots,y_{u_m}]$ be the polynomial ring in $n$
variables over $K$ with each $\deg(y_{u_i})=1$ and let $I_\A$ denote
the kernel of the surjective homomorphism
$\varphi_{\A}:S_\A\rightarrow K[\A]$ defined by
$\varphi_\A(y_{u_i})=u_i$ for all $1\leq i\leq m$. $I_\A$ and
$K[\A]$ are, respectively, called \emph{toric ideal} and \emph{toric ring} of
$\A$.

Firstly, we recall the concept of \emph{combinatorial pure subring}
of a toric ring, introduced in \cite{OhHeHi}, which we will use in
the rest of the paper. Let $T\subseteq [n]:=\{x_1,\ldots,x_n\}$. If $T$ is
a nonempty subset of $[n]$, then we set $\A_T:=\A\cap K[\{x_i:x_i\in T\}]$. A subring of $K[\A]$ of the
form $K[\A_T]$ with $\emptyset\neq T\subseteq [n]$ is called a
combinatorial pure subring of $K[\A]$. For
$\A_T=\{u_{i_1},\ldots,u_{i_r}\}$, we set
$S_{\A_T}=\{y_{u_{i_1}},\ldots,y_{u_{i_r}}\}$. Therefore
$I_{\A_T}=I_{\A}\cap S_{\A_T}$.

\begin{lem}\label{comb pure sub}
Let $\alpha\succeq\beta$ be two $n$-tuple vectors in $\NN^n$. Then
$K[\A^\beta]$ is a combinatorial pure subring of $K[\A^\alpha]$.
\end{lem}
\begin{proof}
Let $\alpha=(k_1,\ldots,k_n)$ and $\beta=(l_1,\ldots,l_n)$. Set
$T=\{x_{11},\ldots,x_{1l_1},\ldots,x_{n1},\ldots,x_{nl_n}\}$. It is
clear that $(\A^\alpha)_T=\A^\beta$.
\end{proof}

\begin{rem}
In \cite{OhHeHi} the authors showed that if $\A$ is a homogeneous
affine semigroup ring generated by monomials belonging to a
polynomial ring with the toric ideal $I$ which is normal, Golod,
Koszul, strongly Koszul, sequentially Koszul or extendable
sequentially Koszul, then any of its combinatorial pure subrings, as
$B$, inherits each of these properties. Moreover, if $\G$ is any
reduced Gr\"{o}bner basis of $I$ then $\G\cap B$ is the reduced
Gr\"{o}bner basis of $J$, where $J$ is the toric ideal of $B$.

Lemma \ref{comb pure sub} implies that $K[\A]$ is a combinatorial pure subring of $K[\A^\alpha]$ and so if $K[\A^\alpha]$ has one of the above properties, then $K[\A]$ has the same property, too.
\end{rem}

In the following we investigate some algebraic properties for
$K[\A^\alpha]$ when they hold for $K[\A]$.

\subsection*{Gr\"{o}bner basis}

Proposition 1.1 of \cite{OhHeHi} together with Lemma \ref{comb pure
sub} guarantees the following:

\begin{prop}\label{oHH Grob}
Let $\alpha\in\NN^n$ and let $\A=\{u_1,\ldots,u_m\}$ be a set of
monomials belonging to $S=K[x_1,\ldots,x_n]$. If $\G$ is the reduced
Gr\"{o}bner basis of $I_{\A^\alpha}$ with respect to a term order
$<$ on $S_{\A^\alpha}$, then $\G\cap K[\A]$ is the reduced
Gr\"{o}bner basis of $I_\A$ with respect to a term order induced by
$<$ on $S_\A$.
\end{prop}

Let $\A=\{u_1,\ldots,u_m\}$ be a set of monomials belonging to
$S=K[x_1,\ldots,x_n]$. Define the term order ``$<^\sharp_{\lex}$''
on the variables of $\{y_{u_1},\ldots,y_{u_m}\}$ in the following
form:
\begin{center}
$y_u<^\sharp_\lex y_v$ $\Leftrightarrow$ $u<_\lex v$ and $y_u=y_v$ $\Leftrightarrow$ $u=v$.
\end{center}
Also, consider the ordering $<_\Lex$ induced by
$$x_{11}>\ldots>x_{1k_1}>\ldots>x_{n1}>\ldots>x_{nk_n}$$ on the
monomials of $\A^\alpha$ for $\alpha=(k_1,\ldots,k_n)$. Again, let
``$<^\sharp_\Lex$'' be a term order on the variables of
$\{y_{u'}:u'\in\A^\alpha\}$ in the following form:
\begin{center}
$y_{u'}<^\sharp_\Lex y_{v'}$ $\Leftrightarrow$ $u'<_\Lex v'$ and $y_{u'}=y_{v'}$ $\Leftrightarrow$ $u'=v'$.
\end{center}

\begin{lem}\label{labelling}
Let $\A$ be a finite set of monomials belonging to
$S=K[x_1,\ldots,x_n]$ and let $\alpha=(k_1,\ldots,k_n)\in\NN^n$. For
$\beta=\alpha+\epsilon_i$ there exists a $K$-algebra isomorphism
$$\varphi:K[(\A^\alpha)^\gamma]\rightarrow K[\A^\beta]$$
where $\gamma=\1+\epsilon_{ik_i}\in\NN^{|\alpha|}$. Here $\1$ is the vector in $\NN^{|\alpha|}$ with all components 1.
\end{lem}
\begin{proof}
For $\alpha$ we have
$[n]^\alpha=\{x_{11},\ldots,x_{1k_1},\ldots,x_{n1},\ldots,x_{nk_n}\}$.
Also,
$$([n]^\alpha)^\gamma=\{x_{111},\ldots,x_{1k_11},\ldots,x_{(i-1)k_{i-1}1},x_{ik_i1},x_{ik_i2},
x_{(i+1)k_{i+1}1},\ldots,x_{n11},\ldots,x_{nk_n1}\}.$$ Consider the
relabeling $\sigma:([n]^\alpha)^\gamma\rightarrow [n]^\beta$ by
$$\sigma(x_{rst})=\left\{
    \begin{array}{ll}
      x_{rs} & \mbox{if}\ t=1 \\
      x_{i(k_i+1)} & \mbox{if}\ t=2.
    \end{array}
  \right.
$$
Then the $K$-algebra homomorphism
$$\varphi:K[(\A^\alpha)^\gamma]\rightarrow K[\A^\beta]$$
defined by $\varphi(u)=\underset{x_{rst}|u}{\prod}\sigma(x_{rst})$
for all monomials $u\in(\A^\alpha)^\gamma$ is an isomorphism.
\end{proof}

\begin{thm}\label{Grob of exp}
Let $\A$ be a set of monomials belonging to $S=K[x_1,\ldots,x_n]$ and let $\alpha=\1+\epsilon_i\in\NN^n$.
If $\G_\A$ is a Gr\"{o}bner basis
of $I_\A$ with respect to a term order induced by $<^\sharp_{\lex}$
on $S_{\A}$, then the Gr\"{o}bner basis of $I_{\A^\alpha}$,
$\G_{\A^\alpha}$, with respect to the term order induced by
$<^\sharp_{\Lex}$ on $S_{\A^\alpha}$ is the union of the set
$$\G_0:=\{y_{u'}y_{v'}-y_{(u'/x_{i1})x_{i2}}y_{(v'/x_{i2})x_{i1}}:\ u',v'\in\A^\alpha\ \mbox{and}\ x_{i1}|u',x_{i2}|v'\}$$
and the set, call $\G_1$, containing all binomials
$\prod^r_ly_{u'_l}-\prod^s_ly_{v'_l}$ with the property that
$\prod^r_ly_{\pi(u'_l)}-\prod^s_ly_{\pi(v'_l)}\in\G_\A$ and $\prod^r_l
u'_l=\prod^s_l v'_l$ for $u'_l,v'_l\in\A^\alpha$.
\end{thm}
\begin{proof}
It is known that toric ideals are
binomial. To showing that $\G_0\cup\G_1$ is a Gr\"{o}bner basis of
$I_{\A^\alpha}$, suppose $0\neq f'=\prod^r_ly_{u'_l}-\prod^s_ly_{v'_l}\in
I_{\A^\alpha}$ and for all $l$ and $m$ with $l\neq m$, $u'_l\neq v'_m$. We want to show that $\inlLs(f')$ is divided by $\inlLs(g')$ for some $g'\in\G_0\cup\G_1$. Let $f:=\gamma(f')\in I_\A$. We have two cases:

\textbf{Case 1.} Assume that $f=0$. Then $r=s$ and we can assume that $\pi(u'_l)=\pi(v'_l)$ for
all $l$.

If $x_i\nmid \pi(u'_l)$, for some $l$, then
$u'_l=v'_l$ which contradicts the assumption. Thus $x_i|\pi(u'_l)$ for all $l$.

Assume that $\inlLs(f')=\prod^r_ly_{u'_l}$. Note that there are two distinct indexes $l_1$ and $l_2$ such that $x_{i1}|u'_{l_1}$ and $x_{i2}|u'_{l_2}$. Otherwise, we will have $u'_l=v'_l$ for all $l$, which is not true. Suppose that $u'_r\leq_\Lex\ldots\leq_\Lex u'_2\leq_\Lex u'_1$ and $l_1=1$. Furthermore, we can assume that $u'_{l_2}$ is a monomial such that $\nu_{x_{i2}}(v'_{l_2})<\nu_{x_{i2}}(u'_{l_2})$. Otherwise, since $\nu_{x_{i2}}(v'_1)>\nu_{x_{i2}}(u'_1)$, it follows that $\nu_{x_{i2}}(\prod^r_lu'_l)<\nu_{x_{i2}}(\prod^r_lv'_l)$, which is a contradiction.

Suppose that $l_2=2$. Set $h:=y_{u'_1}y_{u'_2}-y_{(u'_1/x_{i1})x_{i2}}y_{(u'_2/x_{i2})x_{i1}}$.
Then $h\in I_{\A^\alpha}$. It is clear that $(u'_1/x_{i1})x_{i2}<_\Lex u'_1$. Also, if $u'_1\leq_\Lex (u'_2/x_{i2})x_{i1}$, then using $(u'_2/x_{i2})x_{i1}\leq_\Lex v'_2$ we obtain $u'_1\leq_\Lex v'_2$ and so $\inlLs(f')=\prod^r_ly_{v'_l}$, a contradiction. Therefore $u'_1>_\Lex (u'_2/x_{i2})x_{i1}$. This implies that $\inlLs(h)=y_{u'_1}y_{u'_2}$. In particular, $\inlLs(h)|\inlLs(f')$ and $h\in\G_0$.

\textbf{Case 2.} Assume that $f\neq 0$.
Let $\pi(u'_l)=u_l$, $\pi(v'_l)=v_l$ and
$\inlls(f)=\prod^r_l y_{u_l}$. Since $f\in I_\A$, there exists $g=\prod^p_l y_{u_l}-\prod^q_l y_{w_l}\in\G_\A$ with $\inlls(g)=\prod^p_ly_{u_l}|\inlls(f)$.

If $\inlLs(f')=\prod^r_ly_{u'_l}$, then we can choose some monomials $w'_l$ from
$\A^\alpha$, with the property that $\pi(w'_l)=w_l$, $\prod^p_l u'_l=\prod^q_l w'_l$ and
$\inlLs(g')=\prod^p_l y_{u'_l}$. Set $g':=\prod^p_l y_{u'_l}-\prod^q_l y_{w'_l}$. Then $\inlLs(g')|\inlLs(f')$ and $g'\in\G_1$.

If $\inlLs(f')=\prod^s_ly_{v'_l}$, then it follows from $\inlls(f)=\prod^r_l y_{u_l}$ that there are
two (possibly equal) monomials $v'_h$ and $v'_k$ such that $x_{i1}|v'_h$
and $x_{i2}|v'_k$. More precisely, let $v'_s\leq_\Lex\ldots\leq_\Lex v'_2\leq_\Lex v'_1$ and
$u_r\leq_\lex \ldots\leq_\lex u_2\leq_\lex u_1$. Since $u_1\geq_\lex v_1$ and $v'_1\geq_\Lex u'_1$, it is clear that $x_i$ divides both $u_1$ and $v_1$. If for some $j<i$, $\nu_{x_j}(u_1)>\nu_{x_j}(v_1)$, then $u'_1>_\Lex v'_1$, which is false. Thus suppose that for all $j<i$, $\nu_{x_j}(u_1)=\nu_{x_j}(v_1)$. Also, it is easily to verify that $\nu_{x_i}(u_1)\geq\nu_{x_i}(v_1)$, $\nu_{x_{i1}}(v'_1)>\nu_{x_{i1}}(u'_1)$ and $\nu_{x_{i2}}(v'_1)<\nu_{x_{i2}}(u'_1)$. Especially, since $\prod^r_l u'_l=\prod^s_l v'_l$, it follows that there is some monomial, call $v'_k$, such that $x_{i2}|v'_k$.

If $v'_k$ has a property that $v'_1\geq_\Lex x_{i1}(v'_k/x_{i2})$, then by setting
$f'':=y_{v'_1}y_{v'_k}-y_{x_{i2}(v'_1/x_{i1})}y_{x_{i1}(v'_k/x_{i2})}$ we will have $\inlLs(f'')|\inlLs(f')$. Since $f''\in\G_0$, the assertion is completed. Hence assume that for every $v'_l$ with $l\geq 2$, if $x_{i2}|v'_l$ then $v'_1<_\Lex x_{i1}(v'_l/x_{i2})$. In particular, for such $v'_l$'s, we will have $\nu_{x_{i1}}(v'_l)=\nu_{x_{i1}}(v'_1)-1$ and $\nu_{x_{i2}}(v'_l)\geq\nu_{x_{i2}}(v'_1)+1$.

Since $f\in I_\A$, there exists $g=y_{u_{t_1}}\ldots y_{u_{t_p}}-y_{w_1}\ldots y_{w_q}\in\G_\A$ such that
$\inlls(g)=y_{u_{l_1}}\ldots y_{u_{t_p}}|\inlls(f)$.
Set $h:=\prod^s_ly_{v_l}-(y_{w_1}\ldots y_{w_q}\prod_{l\neq t_i}y_{u_l})$.
\begin{enumerate}[\upshape (1)]
  \item If $\prod^s_ly_{v_l}=y_{w_1}\ldots y_{w_q}\prod_{l\neq t_i}y_{u_l}$ then, by assuming
  $w_1=v_{s_1},\ldots,w_q=v_{s_q}$ and setting $g':=y_{u'_{t_1}}\ldots y_{u'_{t_p}}-y_{v'_{s_1}}\ldots y_{v'_{s_q}}$,
  we will have $\inlLs(g')|\inlLs(f')$ and especially $g'\in\G_1$.
  \item If $\prod^s_ly_{v_l}>^\sharp_\lex y_{w_1}\ldots y_{w_q}\prod_{l\neq t_i}y_{u_l}$, then set
  $h'=\prod^s_ly_{v'_l}-(y_{w'_1}\ldots y_{w'_q}\prod_{l\neq t_i}y_{u'_l})$. Thus there exists
  $g_0=y_{v_{s_1}}\ldots y_{v_{s_p}}-y_{z_1}\ldots y_{z_k}$ such that
  $\inlls(g_0)=y_{v_{s_1}}\ldots y_{v_{s_p}}|\inlls(h)=\inlls(f)$. Again set
  $g'_0=y_{v'_{s_1}}\ldots y_{v'_{s_p}}-y_{z'_1}\ldots y_{z'_k}$. Now
  $\inlls(g'_0)=y_{v'_{s_1}}\ldots y_{v'_{s_p}}|\inlls(f')$ and $g'_0\in \G_1$.
  \item If $\prod^s_ly_{v_l}<^\sharp_\lex y_{w_1}\ldots y_{w_q}\prod_{l\neq t_i}y_{u_l}$, then there is
  $g_1=y_{w_{h_1}}\ldots y_{w_{h_k}}\prod_{l_j\neq t_i}y_{u_{l_j}}-y_{z_1}\ldots y_{z_s}\in\G_\A$ such that
  $\inlls(g_1)=y_{w_{h_1}}\ldots y_{w_{h_k}}\prod_{l_j\neq t_i}y_{u_{l_j}}|\inlls(h)$. Set
  $$h_1:=\prod^s_ly_{v_l}-(y_{z_1}\ldots y_{z_s}\prod_{l\neq l_j,l\neq t_i}y_{u_l}\prod_{l\neq h_i}y_{w_l}).$$
  Now if
  $\prod^s_ly_{v_l}\geq^\sharp_\lex y_{z_1}\ldots y_{z_s}\prod_{l\neq l_j,l\neq t_i}y_{u_l}\prod_{l\neq h_i}y_{w_l}$,
  then by (1) and (2) the assertion is completed. Otherwise, we again go back to (3).
\end{enumerate}

Using the above procedure, after only finitely many steps, we obtain $g_k=\prod^a_ly_{v_{t_l}}-\prod^b_ly_{x_l}\in\G_\A$
with $\inlls(g_k)=\prod^a_ly_{v_{t_l}}$. Now set $g'_k:=\prod^a_ly_{v'_{t_l}}-\prod^b_ly_{x'_l}$. Then
$\inlLs(g'_k)|\inlLs(f')$ and $g'_k\in\G_1$, as desired.
\end{proof}

\begin{exam}
Consider a set
$$\A=\{x^2_1x_2,x_1x_2x_4,x_1x_3,x_2x^2_3,x_2x^2_4\}$$
of monomials belonging to $K[x_1,\ldots,x_4]$. Using CoCoA we obtain
$$\G_\A=\{y_{x^2_1x_2}y_{x_2x^2_4}-y^2_{x_1x_2x_4}\}.$$
Let $\alpha=(1,1,1,2)$. With the same notations of Theorem \ref{Grob of exp} we have
$\G_{\A^\alpha}=\G_0\cup\G_1$ where
$$\G_0= \{y_{x_1x_2x_{41}}y_{x_2x_{41}x_{42}}-y_{x_1x_2x_{42}}y_{x_2x^2_{41}},\
y_{x_2x^2_{41}}y_{x_2x^2_{42}}-y^2_{x_2x_{41}x_{42}},\
y_{x_1x_2x_{41}}y_{x_2x^2_{42}}-y_{x_1x_2x_{42}}y_{x_2x_{41}x_{42}}\}$$
and
$$\G_1= \{y_{x^2_1x_2}y_{x_2x^2_{41}}-y^2_{x_1x_2x_{41}},\
y_{x^2_1x_2}y_{x_2x^2_{42}}-y^2_{x_1x_2x_{42}},\
y_{x^2_1x_2}y_{x_2x_{41}x_{42}}-y_{x_1x_2x_{41}}y_{x_1x_2x_{42}}\}.$$
\end{exam}

Combining Proposition \ref{oHH Grob}, Lemma \ref{labelling} and Theorem \ref{Grob of exp} we have the following.

\begin{cor}
For a given $\alpha\in\NN^n$ and a set $\A$ of monomials in $S$, the Gr\"{o}bner basis of $I_\A$ is consists of quadratic binomials if and only if the Gr\"{o}bner basis of $I_{\A^\alpha}$ has a such construction.
\end{cor}

\subsection*{Sortable sets}

The concept of sortable sets was introduced by Sturmfels \cite{St}. Let $\A$ be a set of monomials in $S$ and let $u,v\in\A$. Write $uv=x_{i_1}\ldots x_{i_{2d}}$ with
$x_{i_1}\leq x_{i_2}\leq\ldots\leq x_{i_{2d}}$. Set $u'=\prod^d_j
x_{2j-1}$ and $v'=\prod^d_j x_{2j}$. We define
$$\s:\A\times\A\rightarrow M_d\times M_d$$
with $\s((u,v))=(u',v')$ where $M_d$ denotes the set of all vectors
in $\ZZ^n$ of modulus $d$. $\A$ is called \emph{sortable}, if $\im(\s)\subseteq \A\times\A$.

For $\alpha\in\NN^n$, we denote by $M^\alpha_d$ the set of all
vectors in $\ZZ^{|\alpha|}$ of modulus $d$.  We define
$$\s^\alpha:\A^\alpha\times\A^\alpha\rightarrow M^\alpha_d\times M^\alpha_d$$
with $\s^\alpha((u,v))=(u',v')$

\begin{thm}\label{sort}
Let $\alpha\in\NN^n$. Let $\A=\{u_1,\ldots,u_m\}$ be a set of
monomials belonging to $S=K[x_1,\ldots,x_n]$. $\A$ is a sortable set
if and only if $\A^\alpha$ has a such property.
\end{thm}
\begin{proof}
Let $(u',v')\in\im(\s^\alpha)$. Then there exists
$(u,v)\in\A^\alpha\times\A^\alpha$ such that
$\s^\alpha((u,v))=(u',v')$. It is easy to see that
$\s((\pi(u),\pi(v)))=(\pi(u'),\pi(v'))$. Now since $\A$ is
sortable, we have that $(\pi(u'),\pi(v'))\in\A\times\A$.
Therefore $(u',v')\in\A^\alpha\times\A^\alpha$.

For the converse direction, suppose that $(u'_0,v'_0)\in\im(\s)$. Then there
exists $(u_0,v_0)\in\A\times\A$ such that
$\s((u_0,v_0))=(u'_0,v'_0)$. Set $u=\prod_{x_i|u_0} x_{1i}$ and
$v=\prod_{x_i|v_0} x_{1i}$. It is clear that
$\s^\alpha((u,v))=(u',v')$ which $u'=\prod_{x_i|u'_{1i}} x_{1i}$ and
$v'=\prod_{x_i|v'_0} x_{1i}$. Now since $\A^\alpha$ is sortable,
so $(u',v')\in\A^\alpha\times\A^\alpha$. This implies that
$(u'_0,v'_0)\in\A\times\A$, as desired.
\end{proof}

By a result due to Sturmfels \cite{St}, toric ideal associated to a
sortable set $\A$ has a Gr\"{o}bner basis consisting of the sorting
relations $y_uy_v-y_{u'}y_{v'}$ with $u,v\in\A$ and
$(u',v')=\s((u,v))$. This result together with Theorem \ref{sort} follows that:

\begin{cor}
Let $\alpha\in\NN^n$ and let $\A$ be a set of monomials belonging to
$S$. If $\A$ (resp. $\A^\alpha$) is sortable then $I_{\A^\alpha}$
(resp. $I_{\A}$) has a Gr\"{o}bner basis consisting of the quadratic
sorting relations.
\end{cor}

\subsection*{Normalness}

For $\A=\{\x^{\u_1},\ldots,\x^{\u_m}\}$ a set of monomials belonging to $S$, we set
$\bar{\A}:=\log_\x(\A)=\{\u_1,\ldots,\u_m\}\subset\ZZ^n_+$.

\begin{thm}\label{normal}
Let $\alpha\in\NN^n$. Let $\A=\{u_1,\ldots,u_m\}$ be a set of
monomials belonging to $S=K[x_1,\ldots,x_n]$. Then $K[\A]$ is a
normal ring if and only if $K[\A^\alpha]$ has this property.
\end{thm}
\begin{proof}
By Theorem 6.1.4 of \cite{BrHe}, $K[\A]$ is normal if and only
$\bar{\A}$ is a normal affine semigroup, i.e. if $t\u\in\bar{\A}$
for some $\u\in\ZZ\bar{\A}$ and $t\in\NN$, then $\u\in\bar{\A}$.
Here $\ZZ\bar{\A}$ means a smallest group containing $\bar{\A}$.
Thus it suffices to show that $(\bar{\A})^\alpha$ is a normal
semigroup when $\bar{\A}$ is normal.

``Only if part'': Let $t\u\in (\bar{\A})^\alpha$ for some
$\u\in\ZZ (\bar{\A})^\alpha$ and $t\in\NN$. It follows that
$t\pi_0(\u)\in\bar{\A}$ and $\pi_0(\u)\in\ZZ\bar{\A}$. Let
$\u=(u_{11},\ldots,u_{1k_1},\ldots,u_{n1},\ldots,u_{nk_n})$. Since $\bar{\A}$ is normal, we conclude that $\pi_0(\u)\in\bar{\A}$.
Let $\pi_0(\u)=(u_1,\ldots,u_n)$. Since
$\u\in\ZZ (\bar{\A})^\alpha$, every component of $\u$ is
integer. Now $u_i=\sum^{k_i}_{j=1}u_{ij}$ implies that $\u\in (\bar{\A})^\alpha$, as desired.

``If part'': Let $t\u\in\bar{\A}$ for some $\u\in\ZZ\bar{\A}$ and
$t\in\NN$. Set $\u':=\sum^n_{i}\u(i)\epsilon_{i1}$. It is clear that
$t\u'\in(\bar{\A})^\alpha$ and so $\u'\in\ZZ(\bar{\A})^\alpha$. In
particular, $\u=\pi_0(\u')\in\bar{\A}$.
\end{proof}

\begin{rem}
The if part of Theorem \ref{normal} is a straightforward consequence
of Proposition 1.2 of \cite{OhHeHi} and Lemma \ref{comb pure sub}.
\end{rem}

\subsection*{Some homological relations}

Corollary 2.5 of \cite{OhHeHi} together with Lemma \ref{comb pure
sub} obtain the following result.

\begin{prop}\label{subring CM}
Let $\alpha\in\NN^n$ and let $\A=\{u_1,\ldots,u_m\}$ be a set of
monomials belonging to $S=K[x_1,\ldots,x_n]$. For the graded Betti
numbers of $I_{\A^\alpha}$ and $I_\A$ we have
$$\beta^{K[\A]}_{ij}(I_\A)\leq\beta^{K[\A^\alpha]}_{ij}(I_{\A^\alpha})\quad\mbox{for all}\ i\ \mbox{and}\ j.$$
\end{prop}


\begin{thm}\label{subring dim}
Let $\alpha\in\NN^n$ and let $\A=\{u_1,\ldots,u_m\}$ be a set of
monomials belonging to $S=K[x_1,\ldots,x_n]$. Then the following hold:
\begin{enumerate}[\upshape (a)]
  \item $\dim(K[\A])\leq\dim(K[\A^\alpha])$;
  \item $\depth(K[\A])\leq\depth(K[\A^\alpha])$;
  \item $\pd(K[\A])\leq\pd(K[\A^\alpha])$;
  \item $\reg(K[\A])\leq\reg(K[\A^\alpha])$.
\end{enumerate}
\end{thm}
\begin{proof}
(a) It is easily seen that $\rank(\A)\leq\rank(\A^\alpha)$. On the
other hand, it is known that $\dim(K[\A])=\rank(\A)$ (See \cite[Chapter 6]{BrHe}).
Therefore the desired equality follows immediately.

(b) It is a straightforward consequence of Proposition 2.4 of
\cite{OhHeHi}, Theorem A.3.4 of \cite{HeHi} and Lemma \ref{comb pure
sub}.

(c) and (d) follows from Proposition \ref{subring CM}.
\end{proof}

\begin{rem}
One may ask that for a set $\A$ of monomials if the toric ring
$K[\A]$ is Cohen-Macaulay, is $K[\A^\alpha]$ Cohen-Macaulay, too?
The answer is negative. For instance, consider
$\A=\{x^3_1,x^2_1x_2,x^3_2\}\subset K[x_1,x_2]$ and
$\alpha=(2,2)\in\NN^2$. By using CoCoA, we see that $K[\A]$ is a
Cohen-Macaulay ring but $K[\A^\alpha]$ is not. $K[\A]$ is of
dimension 2, while the dimension and depth of $K[\A^\alpha]$ are,
respectively, 4 and 3.
\end{rem}

\section*{Acknowledgment}

The authors would like to thank the referee for a careful reading of this
note and for valuable comments and corrections. The research of Rahim Rahmati-Asghar was in part supported by a grant from IPM (No. 92130029). The research of Siamak Yassemi was in part supported by a grant from the University of Tehran (No. 6103023/1/014)


\end{document}